\RequirePackage{lineno}
\documentclass[12pt,a4paper]{amsart}

\newcounter{minutes}
\divide\time by 60
\newcounter{hours}
\multiply\time by 60 \addtocounter{minutes}{-\time}

\usepackage{amsfonts}
\usepackage{amsmath}
\usepackage{amssymb}
\usepackage{amscd,amssymb,amsthm}
\usepackage{verbatim}
\usepackage{graphicx}
\usepackage{ifthen}
\usepackage{hyperref}
\usepackage[T1]{fontenc}
\usepackage{color}
\usepackage{float}

\numberwithin{equation}{section}
\swapnumbers

\theoremstyle{plain}
\newtheorem{theorem}[equation]{Theorem}
\newtheorem{lemma}[equation]{Lemma}

\newtheorem{example}[equation]{Example}
\newtheorem{corollary}[equation]{Corollary}

\newtheorem{remark}[equation]{Remark}

\newtheorem{nonsec}[equation]{}

\newcommand{\beqn}{\begin{equation}}
\newcommand{\eeqn}{\end{equation}}
\newcommand{\bthm}{\begin{theorem}}
\newcommand{\ethm}{\end{theorem}}
\newcommand{\bpf}{\begin{proof}}
\newcommand{\epf}{\end{proof}}
\newcommand{\blem}{\begin{lemma}}
\newcommand{\elem}{\end{lemma}}
\newcommand{\bcol}{\begin{corollary}}
\newcommand{\ecol}{\end{corollary}}
\newcommand{\brmk}{\begin{remark}}
\newcommand{\ermk}{\end{remark}}

\font\ff=eusm10 scaled 1200
\def\K{\hbox{\ff K}}

\newcommand{\R}{\mathbb{R}}
\newcommand{\BB}{\mathbb{B}^2}
\newcommand{\RR}{\mathbb{R}^2}
\newcommand{\Rn}{\mathbb{R}^n}
\newcommand{\MRn}{\overline{\mathbb{R}}^n}
\newcommand{\Bn}{\mathbb{B}^n}
\DeclareMathOperator{\arctanh}{arth}

\newcommand{\Idk}{{\rm Id}_K}

\date{}

\title{Distortion of quasiconformal mappings with identity boundary values}

\author{Matti Vuorinen}
\author{Xiaohui Zhang }

\address{Department of Mathematics and Statistics, University of Turku, 20014 Turku,
Finland} \email{vuorinen@utu.fi, xiazha@utu.fi}

\begin{document}

\def\thefootnote{}
\footnotetext{ \texttt{\tiny File:~\jobname .tex,
          printed: \number\year-\number\month-\number\day,
          \thehours.\ifnum\theminutes<10{0}\fi\theminutes}
} \makeatletter\def\thefootnote{\@arabic\c@footnote}\makeatother

\maketitle



\begin{abstract}
Teichm\"uller's classical mapping problem for plane domains
concerns  finding a lower bound for the maximal dilatation of a quasiconformal homeomorphism
which holds the boundary pointwise fixed, maps the domain onto itself,
and maps a given point of the domain to another given point of the domain.
For a domain $D \subset {\mathbb R}^n\,,n\ge 2\,,$ we consider the class of all $K$-
quasiconformal maps of $D$ onto itself with identity boundary values
and Teichm\"uller's problem in this context.
 Given a map $f$ of this class and a point $x\in D\,,$ we show that the maximal dilatation of $f$ has a lower bound in terms of the distance of $x$ and $f(x)$. We improve recent results for the unit ball and consider this problem in other more general domains. For instance, convex domains, bounded domains and domains with uniformly perfect boundaries are studied.
\end{abstract}





\section{Introduction}

Teichm\"uller's classical mapping problem for plane domains
concerns  finding a lower bound for the maximal dilatation of a quasiconformal homeomorphism
which holds the boundary pointwise fixed, maps the domain onto itself,
and maps a given point of the domain to another given point of the domain (see \cite{av,k,mv,t}).
G.J.~Martin \cite{m2} has recently studied the Teichm\"uller  problem for the mean distortion.
The classical problem has found applications in the theory of homogeneity of domains as introduced in \cite{gp} and more recently in the homogeneity constants of surfaces \cite{bbcmt,bcmt,km}.

Let $D$ be a proper subdomain of $\Rn$ $(n\geq2)$, and let
$$
{\rm Id}_K(\partial D)=\{f:\Rn\to\Rn\,\,\mbox{ is $K$-quasiconformal}: f(x)=x,\, \forall x\in\Rn\setminus D\}.
$$
In his classical work \cite{t} O. Teichm\"uller studied the class $\Idk(\partial D)$ with $D=\RR\setminus\{0,e_1\}, e_1=(1,0)\,,$ and proved the following sharp inequality
$$
\rho_D(x,f(x))\leq\log K
$$
for all $x\in D$, where $\rho_D$ is the hyperbolic metric of $D=\RR\setminus\{0,e_1\}$. This result may be regarded as a stability result since it says that $f(x)$ is contained in the closure of the  hyperbolic ball $B_{\rho_D}(x,\log K)$ centered at the point $x$ and with the radius $\log K$.  In particular, the radius tends to 0 as $K\to1$.

J. Krzy\.z \cite{k} considered the same problem for the case of the unit disk, and  G. D. Anderson and M. K. Vamanamurthy \cite{av} found  a counterpart for Krzy\.z's result in the case of the unit ball in ${\mathbb R}^n, n\ge 3,$ under an additional symmetry hypothesis. Very recently, V. Manojlovi\'c and M. Vuorinen \cite{mv} removed the extra symmetry hypothesis and proved the following Theorem \ref{mvthm}. Continuing
the work of \cite{mv}, we study the case of subdomains of ${\mathbb R}^n\,$ more general than the unit ball. For basic information on quasiconformal maps in ${\mathbb R}^n\,,n \ge 2\,,$ we refer the reader to \cite{fm,lv,va}.

 As in \cite[p.~97 (7.44), p.~138, Theorem 11.2]{vu88}, we denote by $\varphi_{K,n}:[0,1]\to[0,1], K>0,$  the
 special function connected with the Schwarz lemma. Its precise definition of $\varphi_{K,n}$ is given in (\ref{phi}). What is  important is that it is an increasing homeomorphism with $\varphi_{K,n}(r)\to r$ as $K\to 1\,.$
 Write $\Bn(r)=\{x\in{\mathbb{R}}^n\,:\,|x|<r\}$ and $\Bn =\Bn(1)$.

\bthm\cite[Theorem 1.9]{mv}\label{mvthm}
If $f\in\Idk(\partial \Bn)$, then, for all $x\in\Bn$,
$$
  \rho_{\Bn}(x,f(x))\leq\log\dfrac{1-\varphi_{1/K,n}(1/\sqrt2)^2}{\varphi_{1/K,n}(1/\sqrt2)^2}
$$
where $\rho_{\Bn}$ is the hyperbolic metric defined in (\ref{rho}).
\ethm

Theorem \ref{mvthm} shows that the mapping $f$ uniformly tends to the identity mapping when the maximal dilatation goes to 1.

In this paper we will first prove the following theorem which is similar to the result of Manojlovi\'c and Vuorinen.

\begin{theorem}\label{bnd4ball}
If $f\in\Idk(\partial \Bn)$, then, for all $x\in\Bn$,
$$
  \rho_{\Bn}(x,f(x))\leq\log\frac{1-\varphi_{1/K,n}(1/2)}{\varphi_{1/K,n}(1/2)}.
$$
\end{theorem}

Motivated by a question of F.W. Gehring, J. Krzy\.z \cite[Theorem
1]{k} proved the following theorem. See also Teichm\"uller \cite{t} and Krushkal \cite[p.59]{kr}.

\begin{theorem} \label{krzyz0}
\cite[Theorem 1]{k} For $f\in Id_K(\partial
\BB)$ the sharp bounds are:
\begin{equation}
\label{krzyz1}
|f(0)|\leqslant\mu^{-1}\left(\log\frac{\sqrt{K}+1}{\sqrt{K}-1}\right)\equiv c
\end{equation}
where $\mu$ is the function defined in (\ref{mymu}) and
\begin{equation}
\label{krzyz2} \tanh\frac{\rho_{\BB}(f(z),z)}{2}\leqslant c
\end{equation}
for every $z\in \BB$, where $\rho_{\BB}$ is the hyperbolic metric.
\end{theorem}


A comparison shows that Theorem \ref{bnd4ball} yields a better bound than Theorem \ref{mvthm}  when $n=2$ (see Remark \ref{rmk4ballpf} (3) below).
However, the sharp result of Krzy\.z, which only applies for
$n=2\,,$ is even better when $n=2\,.$ For a graphical
comparison of these bounds for $n=2$, see Figure \ref{fig4rho}.


We next extend this result to the case of convex domains.
To this end, we require a suitable metric, the distance ratio metric
$j_D$ of a domain $D \subset \mathbb{R}^n\,,$ defined in Section 2.

\begin{theorem}\label{thm:bnd4convex}
Let $D\subsetneq\Rn$ be a convex domain and $f\in{\rm Id}_K(\partial D)$. Then, for all $x\in D$,
$$
j_D(x,f(x))\leq \log\left(1+\sqrt{\left(\frac{2\varphi_{K,n}(1/3)}{1-\varphi_{K,n}(1/3)}\right)^2-1}\right).
$$
\end{theorem}

For $K$ close to $1$, the inequality of Theorem \ref{thm:bnd4convex} can be simplified further.

\bthm\label{thm:bnd4convexSK}
 Let $D\subsetneq\Rn$ be a convex domain and
 $$K_n=\left(1+\dfrac{\log 2}{n-1+\log 3}\right)^{n-1}
 \in[K_2,2),\qquad K_2\approx1.33029.$$
 If $K\in(1,K_n]$ and $f\in{\rm Id}_K(\partial D)$,
 then for all $x\in D$
$$
j_D(x,f(x))\leq2\sqrt{1+\log 6}(K-1)^{1/2}.
$$
\ethm

A common feature of all these results, including Teichm\"uller's
original result, is that if $f(x) \neq x$ for some point in the
domain, then the maximal dilatation $K >1\,.$ This lower bound
is not true for all subdomains in $\mathbb{R}^n, n\ge3,$ as is
easy to show by an example, see Remark \ref{mvu84rmk}.

In the case of uniform domains \cite{ms,va88} with connected boundary, M. Vuorinen \cite{vu84} established
\begin{equation} \label{mvu84}
  K\geq c_1(n,D)k_D(x,f(x))^n
\end{equation}
whenever the quasihyperbolic distance $k_D(x,f(x))$ exceeds a bound depending only on $n$ and $D$. Here $c_1(n,D)$ is a positive constant depending only on $n$ and $D$.

The proof of the inequality  \eqref{mvu84} makes use of the classical V\"ais\"al\"a's lower bound for the modulus of the family of curves joining continua \cite[Theorem 10.12]{va}. Aseev's theorem \cite[Theorem 3]{as} (see Lemma \ref{lemma} below) provides a counterpart of this result with continua
replaced with uniformly perfect sets. In this way we can prove that \eqref{mvu84} also holds for the case of uniform domains with uniformly perfect boundary \cite{bp,jv,s}.

\bthm\label{thm4unipfbnd}
Let $D\subsetneq\Rn$ be a uniform domain with uniformly perfect boundary and $x\in D$,
and let $f\in{\rm Id}_K(\partial D)$.
Then there exists a positive constant $c_2(n,D)$ depending only on $n$ and the constants of uniformity and uniform perfectness of the domain $D$ such that
for all $x \in D$
$$
  K\geq c_2(n,D)k_D(x,f(x))^n.
$$
\ethm

The H\"older continuity of quasiconformal self mappings of the unit ball with the origin fixed is an important topic which was first studied by Ahlfors \cite{ah} when
the dimension $n=2$.  Refining Ahlfors' result, A. Mori proved that a $K$-quasiconformal
mapping $f: \BB \to \BB= f(\BB), f(0)=0,$ satisfies for all $x,y\in \BB$ the inequality
$$  |f(x)-f(y)| \le M |x-y|^{1/K}$$
with the best possible constant $M=16$
independent of $K\,$ and the sharp exponent $1/K$ \cite{lv}.
Later on, it was conjectured that here $16$ can be
replaced by $16^{1-1/K}\,.$ This conjecture, sometimes referred to as
the {\it Mori's conjecture for planar quasiconformal maps of the
unit disk}, is a well-known open problem and it has been studied
by many people.
For the higher dimensional case $n \ge 3$ an asymptotically sharp constant, i.e. a constant tending to $1$
when $K \to 1\,,$ was proved for the first time by Fehlmann and Vuorinen \cite{fv}. Very recently, Bhayo and Vuorinen \cite{bv} improved the previous results by using a refined inequality for the Teichm\"uller function and introducing an additional parameter which was chosen in an optimal way. Many authors have studied these questions. For the detailed history of the H\"older continuity of quasiconformal mappings, the readers are referred to the bibliographies of \cite{mrv}, \cite{fv}, \cite{vu88} and \cite{bv}. We will consider this problem and improve the constant for the class of quasiconformal mappings of the unit ball with identity boundary values. Note that in this case it is not required that the origin be fixed by the mapping.

\begin{theorem}\label{holder}
If $f\in{\rm Id}_K(\partial\Bn)$, then for all $x,y\in\Bn$
$$|f(x)-f(y)|\leq M_1(n,K)|x-y|^\alpha,\qquad \alpha=K^{1/(1-n)}$$
where $M_1(n,K)=\lambda_n^{1-\alpha}C(\alpha)$ and $C(\alpha)=2^{1-\alpha}\alpha^{-\alpha/2}(1-\alpha)^{(\alpha-1)/2}\,,$
with $M_1(n,K)\to1\,$ when $K \to 1\,,$ and $\lambda_n\in[4,2e^{n-1})$ is the Gr\"otzsch ring constant.
\end{theorem}

For the planar case of $n=2$, I. Prause \cite{p} has proved that $4^{1-1/K}$ is the optimal constant under the same conditions of Theorem \ref{holder}.


\medskip

\section{Notation and preliminary results}
In this section we shall follow the standard notation and terminology for $K-$quasi\-con\-formal
mappings in the Euclidean $n-$space $\Rn$, see e.g. \cite{avvb}, \cite{va} and \cite{vu88}.

The \emph{hyperbolic metric} $\rho_{\Bn}(x,y)$ on $\Bn$ is defined by
\beqn\label{rho}
   \tanh^2\dfrac{\rho_{\Bn}(x,y)}{2}=\dfrac{|x-y|^2}{|x-y|^2+(1-|x|^2)(1-|y|^2)}.
\eeqn
A simple argument shows that we have
$$
  |x-y|\leq2\tanh\dfrac{\rho_{\Bn}(x,y)}{4}
$$
for all $x,y\in\Bn$ with equality for $x=-y$ (see \cite[(2.27)]{vu88}).

Let $D\subsetneq\R^n$ be a domain. The \emph{quasihyperbolic metric} $k_D$ is defined by \cite{gp}
$$
  k_{D}(x,y)=\inf_{\gamma\in\Gamma}\int_{\gamma}\frac{1}{d(z)}|dz|, {\quad} x,y\in D,
$$
where $\Gamma$ is the family of all rectifiable curves in $D$ joining $x$ and $y$, and
$d(z)=d(z,\partial D)$ is the Euclidean distance between $z$ and the boundary of $D$.
The \emph{distance-ratio metric} or $j-$\emph{metric} is defined as \cite{gp,vu85}
\beqn\label{jmetric}
   j_D(x,y)=\log\left(1+\frac{|x-y|}{\min\{d(x),d(y)\}}\right),\quad x,y\in D.
\eeqn
It is well known that \cite[Lemma 2.1]{gp}, \cite[(3.4)]{vu88}
$$
   j_D(x,y)\leq k_D(x,y)
$$
for all domains $D\subsetneq\Rn$ and $x,y\in D$.

A domain $D$ in $\Rn$, $D\neq\Rn$, is called \emph{uniform},
if there exists a number $U=U(D)\geq1$ such that
$k_D(x,y)\leq U\, j_D(x,y)$ for all $x,y\in D$.
Uniform domains were introduced by Martio and Sarvas \cite{ms}.
Presently, there are several equivalent definitions of uniform domains,
see, for instance, V\"ais\"al\"a \cite{va88}. The above definition, which
is most convenient for the sequel, is adopted
from Gehring and Osgood \cite{go} and  Vuorinen \cite{vu85}.

It is well known \cite[Lemma 7.56]{avvb} that the unit ball $\Bn$ is a uniform domain with the constant $U=2$.

Given $E,F,G\subset \Rn$ we use the notation $\Delta(E,F;G)$ for the family of all curves that join the sets $E$ and $F$ in $G$ and $M(\Delta(E,F;G))$ for its modulus. If $G=\Rn$, we may omit $G$ and simply denote $\Delta(E,F;G)$ by $\Delta(E,F)$. For a ring domain $R(C_0,C_1)$ with complementary components $C_0$ and $C_1$, we define the modulus of $R(C_0,C_1)$ by
$$
 {\rm mod}R(C_0,C_1)=\left(\dfrac{\omega_{n-1}}{M(\Delta(C_0,C_1))}\right)^{1/(n-1)},
$$
where $\omega_{n-1}$ is the surface area of the unit sphere in $\Rn$.

The \emph{Gr\"otzsch ring domain} $R_{G,n}(s)$, $s>1$, and the \emph{Teichm\"uller ring domain} $R_{T,n}(t)$, $t>0$, are doubly connected domains with complementary components $(\overline{\Bn},[se_1,\infty))$ and $([-e_1,0],[te_1,\infty))$, respectively. For their capacities we write
$$
  \left\{
  \begin{array}{ll}
  \gamma_n(s)={\rm cap}R_{G,n}(s)=M(\Delta(\overline{\Bn},[se_1,\infty])),\\
  \tau_n(t)={\rm cap}R_{T,n}(t)=M(\Delta([-e_1,0],[te_1,\infty])).
  \end{array}\right.
$$
These functions are related by the functional identity
$$
  \gamma_n(s)=2^{n-1}\tau_n(s^2-1).
$$

For $K>0$ we define an increasing homeomorphism $\varphi_{K,n}:[0,1]\to[0,1]$ with $\varphi_{K,n}(0)=0$, $\varphi_{K,n}(1)=1$ and
\beqn\label{phi}
  \varphi_{K,n}(r)=\dfrac{1}{\gamma_n^{-1}(K\gamma_n(1/r))},\quad 0<r<1.
\eeqn

The following important estimates are well known \cite{vu88}
\beqn\label{bd4phi}
r^\alpha\leq\varphi_{K,n}(r)\leq\lambda_n^{1-\alpha}r^\alpha\leq2^{1-1/K}Kr^\alpha,\quad \alpha=K^{1/(1-n)}\, ,
\eeqn
\beqn\label{bd4phi2}
2^{1-K}K^{-K}r^\beta\leq\lambda_n^{1-\beta}r^\beta\leq\varphi_{1/K,n}(r)\leq r^\beta,\quad \beta=1/\alpha,
\eeqn
where $K\geq1, r\in(0,1)\,,$ and the constant $\lambda_n\in[4,2e^{n-1})$ is the so-called \emph{Gr\"otzsch ring constant}. In particular, $\lambda_2=4$.

For $n\geq2$, $t\in(0,\infty)$, $K>0$, we denote
\beqn\label{eta}
\eta_{K,n}(t)=\tau_n^{-1}\left(\dfrac1K\tau_n(t)\right)=
\dfrac{1-\varphi_{1/K,n}(1/\sqrt{1+t})^2}{\varphi_{1/K,n}(1/\sqrt{1+t})^2}\,.
\eeqn

For the purpose of comparing our bounds to earlier bounds
we wish to express these functions in terms of well-known
functions. This is possible only for  $n=2$ \cite{avvb}.
and the general case  $n\ge3$ remains as a challenge.
To this end, it is enough to express the formulas for $\gamma_2(s)$
and $\varphi_{K,2}(r)$ in terms of classical special functions.
First, we consider a decreasing
homeomorphism $\mu:(0,1)\longrightarrow(0,\infty)$ defined by \cite{lv,vu88}
\begin{equation}
\label{mymu} \mu(r)=\frac\pi 2\,\frac{{\K}(r')}{{\K}(r)}, \quad {\K}(r)=\int_0^1\frac{dx}{\sqrt{(1-x^2)(1-r^2x^2)}}\, ,
\end{equation}
where ${\K}(r)$ is Legendre's complete elliptic integral
of the first kind and $ r'=\sqrt{1-r^2},$
for all $r\in(0,1)$. Now $\gamma_2(1/r) = 2 \pi/\mu(r), r \in (0,1)$
and $\varphi_{K,2}(r)= \mu^{-1}(\mu(r)/K)\,$ by \cite{lv,vu88}.


Let $\alpha>0$ and assume that $D\subset\overline \Rn$ is a closed set containing at least two points. Then $D$ is $s-$\emph{uniformly perfect} if there is no ring domain separating $D$ with the modulus greater than $s$. $D$ is \emph{uniformly perfect} if it is $s-$uniformly perfect for some $s>0$ \cite{bp}. Uniformly perfectness is a useful tool in many topics of geometric function theory. See \cite{s} for a survey of this topic. The following lemma is an analog of V\"ais\"al\"a's lemma \cite[Theorem 10.12]{va} with continua replaced by uniformly perfect sets.

\blem\cite[Theorem 3]{as}\label{lemma}
  Suppose that $s>0$ and that $s-$uniformly perfect sets $E_0$ and $E_1$ meets each component of the complement of the spherical ring $D=\{x:\, r_1<|x-x_0|<r_2\}\subset\Rn$ with the following relation between the radii $$
  r_2/r_1>1+2e^{s}.
  $$
  Then
  $$
  {\rm cap}(E_0,E_1;D)\geq C\log\dfrac{r_2}{r_1},
  $$
  where the constant $C>0$ depends only on $s$ and the dimension $n$ of the space.
\elem

Note that, from the proof of this lemma, it is easy to see that the result obviously holds if one of the two sets $E_0$ and $E_1$ is a continuum.


\medskip

\section{Proofs of main results}

\begin{nonsec}{\rm {\bf Remark.} \label{mvu84rmk}} {\rm
(1). Let $f: D\to D$ be a quasiconformal mapping which extends to a homeomorphism $f^*: \overline{D}\to\overline{D}$ with $f^*(x)=x$ for all $x\in\partial D$. By \cite[Theorem 1]{rickman} or \cite[Theorem 2]{va75}, the mapping $f$ can be extended to a quasiconformal mapping $\tilde{f}: \Rn\to\Rn$ by setting $\tilde{f}(x)=x$ for $x\not\in D$. Moreover, $\tilde{f}$ has the same dilatation as $f$.

(2). As pointed out in \cite{vu84}, it is not true for $n\ge 3$ that for
$f \in {\rm Id}_K(\partial D)$ the condition $k_D(x,f(x))>0$
implies $K>1\,.$ Indeed, let $X= \{(x,0,0): x \in {\mathbb R} \}$ be the $x_1$-axis, let $D= {\mathbb R}^3 \setminus X \,,$ and  let
$f: D \to D$ be a rotation  around the  $x_1$-axis with
$f(x) =(0,-1,0), x=(0,1,0)\,.$
Then $f$ is conformal, i.e. $K=1\,,$ $f$ keeps the $x_1$-axis $X= \partial D$ pointwise
fixed, and $D$ is a uniform domain with connected boundary $X$ and $k_D(x,f(x))= \pi\,.$ Clearly, for this domain $c_1(3,D)\le 1/ \pi^3 \,.$

(3). For uniformly perfect hyperbolic domains $D\subset\RR$ the hyperbolic metric and the quasihyperbolic metric are equivalent \cite{kl}, i.e. there exists a constant $C(D)>1$ such that
$$
\rho_D\leq k_D\leq C(D)\rho_D.
$$
Let $D$ be a uniformly perfect hyperbolic domain in the plane and $f \in {\rm Id}_K(\partial D)$. Then for all $z\in D$,
$$
j_D(z,f(z))\leq k_D(z,f(z))\leq C(D)\rho_D(z,f(z))\leq C(D)Kr(K),
$$
where $Kr(K)$ is Krzy\.z's bound
$$
Kr(K)=2\arctanh\mu^{-1}\left(\log\frac{\sqrt{K}+1}{\sqrt{K}-1}\right).
$$
}
\end{nonsec}

\begin{figure}[h]
\includegraphics[width=7cm]{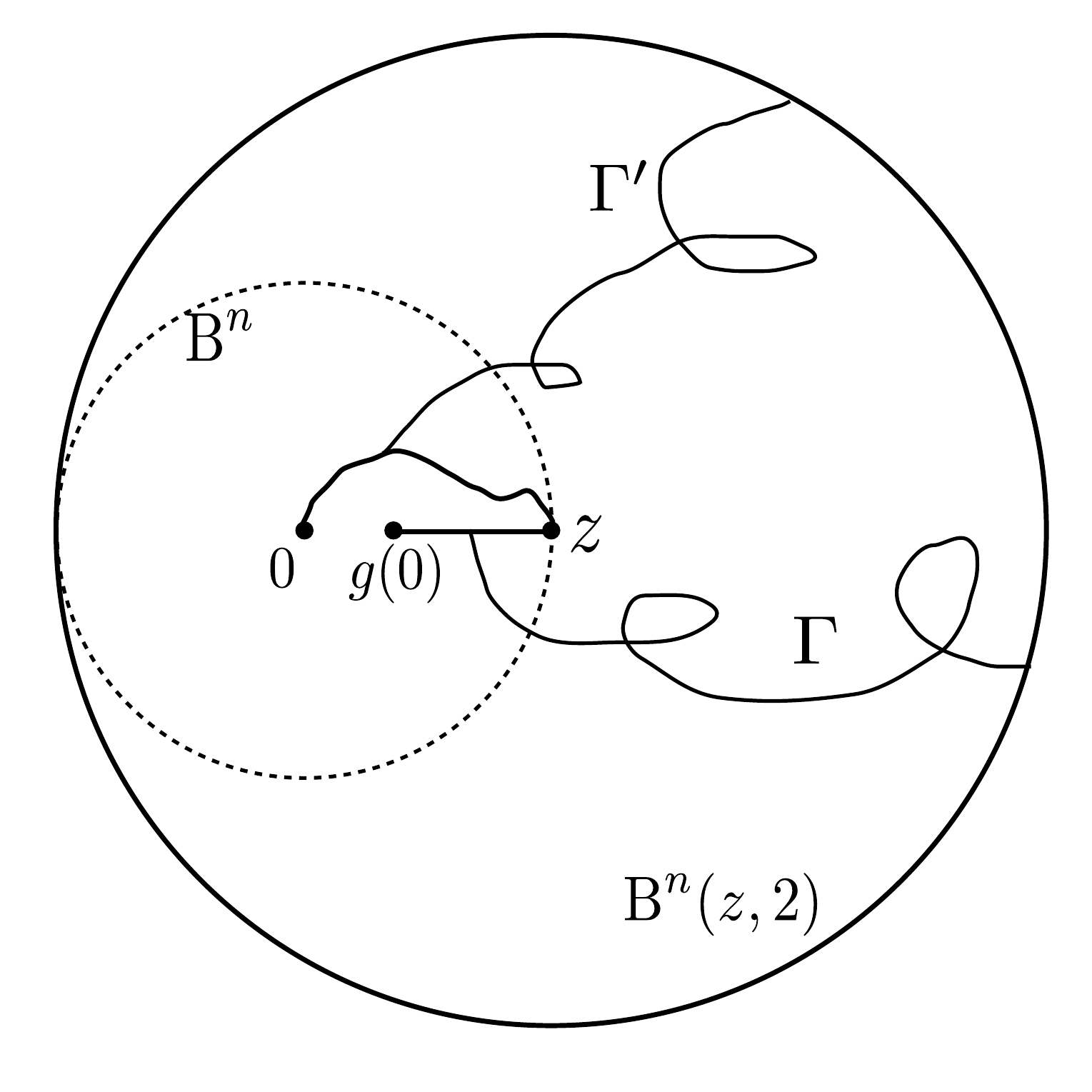}
\caption{The proof of Theorem \ref{bnd4ball} visualized.}
\end{figure}

\begin{nonsec}{\rm {\bf Proof of Theorem \ref{bnd4ball}.}
For arbitrarily given $x\in\Bn$ let $T_x$ be a M\"obius transformation of $\MRn$ with $T_x(\Bn)=\Bn$ and $T_x(x)=0$ \cite{b}.
Define $g:\,\MRn\to \MRn$ by setting $g(z)=T_x\circ f\circ T^{-1}_x(z)$ for $z\in\Bn$ and $g(z)=z$ for $z\in\MRn\setminus\Bn$. Then $g\in\Idk(\partial\Bn)$ with $g(0)=T_x(f(x))$. Since the hyperbolic metric $\rho_{\Bn}$ is preserved under M\"obius transformations of $\Bn$ onto itself, we have that for $x\in\Bn$
\begin{equation}\label{normrho}
\rho_{\Bn}(x,f(x))=\rho_{\Bn}(0,g(0)).
\end{equation}

Choose $z\in\partial \Bn$ such that the point $g(0)$ is contained in the segment $[0,z]$.
Let $\Gamma=\Delta([g(0),z],\partial\Bn(z,2);\Bn(z,2))$ be the  family of curves joining $[g(0),z]$ to $\partial\Bn(z,2)$ in $\Bn(z,2)$, and $\Gamma'=g^{-1}(\Gamma)=\Delta(g^{-1}([g(0),z]),\partial\Bn(z,2);\Bn(z,2))$. Then we have
$$
M(\Gamma)=\gamma_n\left(\frac{2}{1-|g(0)|}\right)
$$
and by the spherical symmetrization with center at $z$
$$
M(\Gamma')\geq\gamma_n(2).
$$
By $K-$quasiconformality we have $K\,M(\Gamma)\geq M(\Gamma')$ \cite{va} which, together with \eqref{phi} implies
$$
|g(0)|\leq 1-2\varphi_{1/K,n}(1/2)
$$
and
$$
\rho_{\Bn}(0,g(0))=\log\frac{1+|g(0)|}{1-|g(0)|}\leq\log\frac{1-\varphi_{1/K,n}(1/2)}{\varphi_{1/K,n}(1/2)}.
$$
\hfill$\Box$
}
\end{nonsec}

\begin{nonsec}{\rm {\bf Remark.} \label{rmk4ballpf}} {\rm
(1). If we take a slightly different construction of Gr\"otzsch ring domain, we will get another form of bound as following argument shows
.
\begin{figure}[H]
\includegraphics[width=9cm]{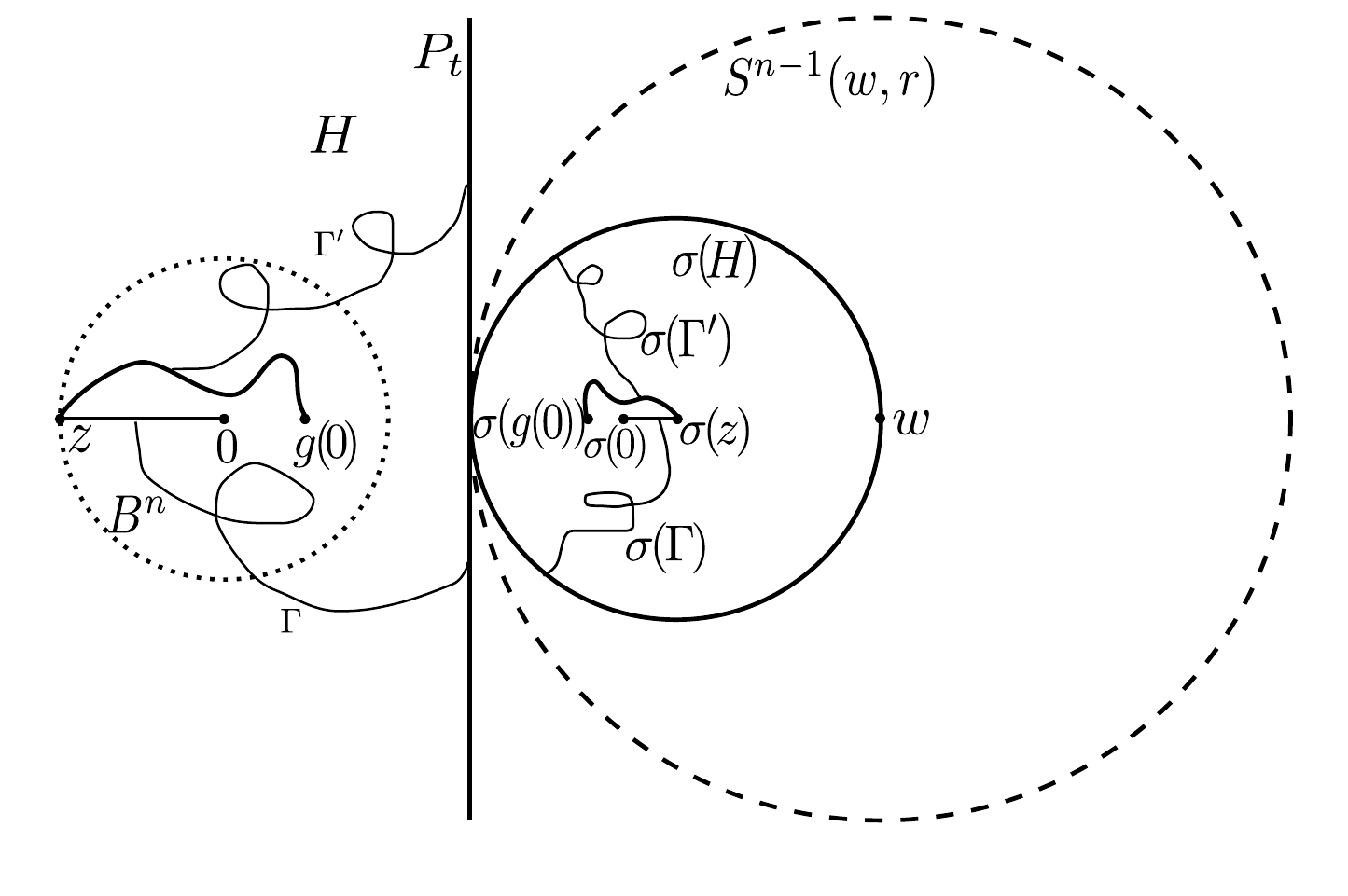}
\caption{The argument visualized.}
\end{figure}

Choose $z\in\partial \Bn$ such that the origin is contained in the segment $[g(0),z]$. For $t\geq0$ let $P_t=P(g(0),(1+t)|g(0)|)=\{x\in\Rn:\, x\cdot g(0)=(1+t)|g(0)|\}$ be the hyperplane in $\MRn$ perpendicular to the vector $g(0)$, at distance $1+t$ from the origin \cite{b}, and the half space $H$ be the component of $\Rn\setminus P_t$ which contains the origin.
Let $\sigma$ be the inversion in the sphere $S^{n-1}(w,r)$ where $w=(3+2t)g(0)/|g(0)|$ and $r=2+t$, then we have
$\sigma(z)=(4+3t)g(0)/(2|g(0)|)$ and $\sigma(H)=\Bn(\sigma(z),r/2)$.
It is easy to see that
$$
|\sigma(z)-\sigma(0)|=\frac{2+t}{6+4t}
$$
and
$$
|\sigma(z)-\sigma(g(0))|=\frac{(2+t)(1+|g(0)|)}{6+4t-2|g(0)|}.
$$
Let $\Gamma=\Delta([0,z],P_t;H)$ be the  family of curves joining $[0,z]$ to $P_t$ in $H$, and $\Gamma'=g(\Gamma)=\Delta(g([0,z]),P_t;H)$. By the conformal invariance of the modulus, we have
$$
M(\Gamma)=M(\sigma(\Gamma))=\gamma_n\left(\frac{r/2}{|\sigma(z)-\sigma(0)|}\right)
$$
and by the spherical symmetrization with center at $\sigma(z)$
$$
M(\Gamma')=M(\sigma(\Gamma'))\geq\gamma_n\left(\frac{r/2}{|\sigma(z)-\sigma(g(0))|}\right).
$$
By $K-$quasiconformality we have $K\,M(\Gamma)\geq M(\Gamma')$ \cite{va} implying
$$
\frac{1+|g(0)|}{3+2t-|g(0)|}\leq\varphi_{K,n}\left(\frac{1}{3+2t}\right),
$$
and further
$$
|g(0)|\leq\frac{(3+2t)\varphi_{K,n}(1/(3+2t))-1}{1+\varphi_{K,n}(1/(3+2t))}.
$$
Since the above inequality holds for all $t\geq0$, the choice $t=0$ gives
$$
|g(0)|\leq\frac{3\varphi_{K,n}(1/3)-1}{1+\varphi_{K,n}(1/3)}
$$
and
$$
\rho_{\Bn}(0,g(0))=\log\frac{1+|g(0)|}{1-|g(0)|}\leq\log\frac{2\varphi_{K,n}(1/3)}{1-\varphi_{K,n}(1/3)}.
$$
Hence we have that
\begin{equation}\label{bnd4ball2}
  \rho_{\Bn}(x,f(x))\leq\log\frac{2\varphi_{K,n}(1/3)}{1-\varphi_{K,n}(1/3)}
\end{equation}
holds for all  $f\in\Idk(\partial \Bn)$ and $x\in\Bn$.
\hfill$\Box$

(2). For the planar case of $n=2$, the inequality \eqref{bnd4ball2} and Theorem \ref{bnd4ball} exactly give the same bound. In fact, by setting $r=1/3$ in the following identity \cite[Theorem 10.5(2)]{avvb}
$$
\varphi_{1/K,2}\left(\frac{1-r}{1+r}\right)=\frac{1-\varphi_{K,2}(r)}{1+\varphi_{K,2}(r)},\quad r\in[0,1],
$$
we have
$$
\log\frac{1-\varphi_{1/K,2}(1/2)}{\varphi_{1/K,2}(1/2)}=\log\frac{2\varphi_{K,2}(1/3)}{1-\varphi_{K,2}(1/3)}.
$$

(3). We claim that Theorem \ref{bnd4ball} yields a better bound than Theorem \ref{mvthm}  when $n=2\,.$ Indeed, from \cite[Theorem 10.9(2)]{avvb} and the proof of \cite[Theorem 10.15]{avvb} it is easy to see that
$$
\varphi_{1/K,2}(r^p)\leq\varphi_{1/K,2}(r)^p\quad \mbox{for}\quad 0<p<1,
$$
which implies that
$$
\log\frac{1-\varphi_{1/K,2}(1/2)}{\varphi_{1/K,2}(1/2)}\leq\log\frac{1-\varphi_{1/K,2}(1/\sqrt2)^2}{\varphi_{1/K,2}(1/\sqrt2)^2}.
$$
}
\end{nonsec}

\begin{figure}[h]\label{fig4rho}
\includegraphics[width=9cm]{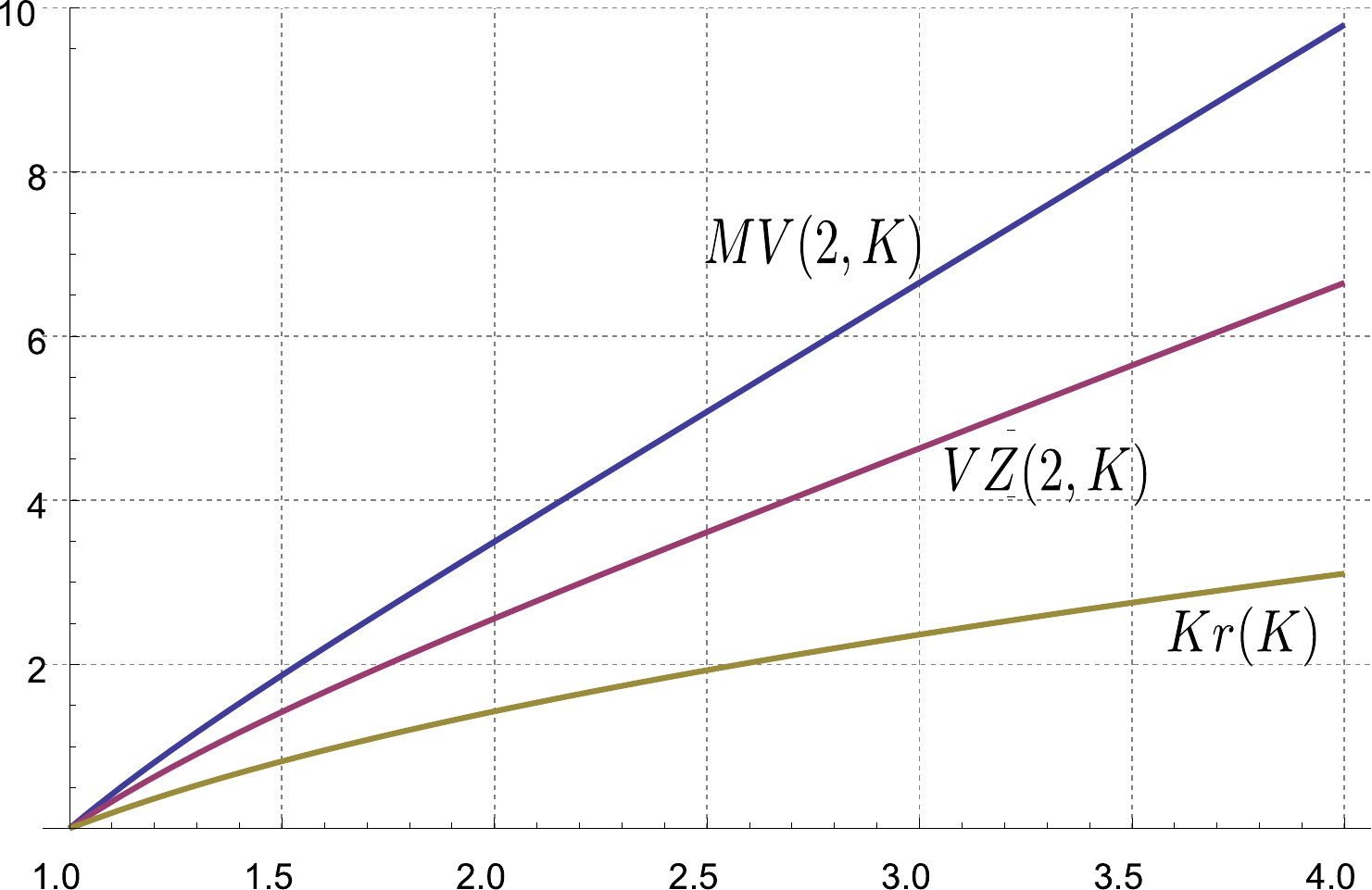}
\caption{\small Graphical comparison of bounds when $n=2$ as a function of $K$: (a) the Manojlovi\'c and Vuorinen bound from
Theorem \ref{mvthm},
$$
MV(2,K)=\log\frac{1-\varphi_{1/K,2}(1/\sqrt2)^2}{\varphi_{1/K,2}(1/\sqrt2)^2}.
$$
(b) the bounds from Theorem \ref{bnd4ball},
$$
VZ(2,K)=\log\frac{1-\varphi_{1/K,2}(1/2)}{\varphi_{1/K,2}(1/2)}.
$$
(c) the Krzy\.z bound from Theorem \ref{krzyz0} (valid only for $n=2$),
$$
Kr(K)=2\arctanh\mu^{-1}\left(\log\frac{\sqrt{K}+1}{\sqrt{K}-1}\right).
$$}
\end{figure}

Note that Vuorinen's example related to
\eqref{mvu84} in Remark \ref{mvu84rmk} is unbounded. For bounded domain $D\subset\Rn$, however, we have following estimate.

\bthm\label{thm4bdomain}
Let $D$ be a bounded domain in $\Rn$, and $f\in\Idk(\partial D)$. Then for all $x\in D$
$$
|f(x)-x|\leq{\rm diam}(D)\tanh\left(\dfrac{1}{2}\log\dfrac{1-b}{b}\right),\quad b=\varphi_{1/K,n}(1/2) \,.
$$
\ethm

\bpf
 For $x\in D$, $D\subset\Bn(x,{\rm diam}(D))$ since $D$ is bounded. Let $g(w)=(w-x)/{\rm diam}(D)$. It is easy to see that $h=g\circ f\circ g^{-1}\in\Idk(\partial\Bn)$. Hence it follows from Theorem \ref{bnd4ball} that
 $$
   \rho_{\Bn}\left(\dfrac{f(x)-x}{{\rm diam}(D)},0\right)=\rho_{\Bn}(h(0),0)\leq\log\dfrac{1-b}{b},
 $$
 and hence
 $$
   |f(x)-x|\leq{\rm diam}(D)\tanh\left(\dfrac{1}{2}\log\dfrac{1-b}{b}\right)
 $$
 since $\rho_{\Bn}(z,0)=2\arctanh|z|$ for $z\in\Bn$.
\epf

\begin{example}
{\rm
For sufficiently small $\epsilon\in (0,1)\,$ there exists $f\in{\rm Id}_{1+\epsilon}(\partial(\Bn\setminus\{0\}))$ and $x_0\in\Bn\setminus\{0\}$ such that $k_{\Bn\setminus\{0\}}(x_0,f(x_0))=1/\epsilon.$ Actually, we can take $f$
to be the radial mapping $f(x)=|x|^{a-1}x$ with $a=(1+\epsilon)^{1/(n-1)}$, and $x_0=e^{-b}e_1$ with $b=1/(\epsilon(a-1))<\log2.$ Then $K(f)=a^{n-1}=1+\epsilon$ (see \cite[16.2]{va}). It is clear that $|x_0|,|f(x_0)|<1/2$, and hence $k_{\Bn\setminus\{0\}}(x_0,f(x_0))=k_{\Rn\setminus\{0\}}(x_0,f(x_0))=\log(|x_0|/|f(x_0)|)=1/\epsilon.$
}
\end{example}

\begin{nonsec}{\rm {\bf Proof of Theorem \ref{thm:bnd4convex}.}
We may assume that $d(x,\partial D)\leq d(f(x),\partial D)$ since $f^{-1}$ is also in $\Idk(\partial D)$.
Let $z\in\partial D$ with $d(x,\partial D)=|x-z|$. For $t>0$ let $P_t$ be the hyperplane perpendicular to $x-z$ and at distance $t$ from the point $z$, and the half space $H$ be the component of $\Rn\setminus P_t$ which contains $z$. Let $\sigma$ be the inversion in the sphere $S^{n-1}(w,t)$ where $w=z+2t(z-x)/|z-x|$, then we have
$\sigma(H)=\Bn(\sigma(z),t/2)$. It is easy to see that
$$
|\sigma(z)-\sigma(x)|=\left|\frac{t}{2}-\frac{t^2}{2t+|x-z|}\right|
$$
and
$$
|\sigma(z)-\sigma(y)|=\left|\frac{t}{2}-\frac{t^2}{2t+|f(x)-z|}\right|
$$
where $y=z+|f(x)-z|(x-z)/|x-z|$.

Let $\Gamma=\Delta([x,z],P_t;H)$ be the  family of curves joining $[x,z]$ to $P_t$ in $H$, and $\Gamma'=f(\Gamma)=\Delta(f([x,z]),P_t;H)$. By the conformal invariance of the modulus and the spherical symmetrization with center at $z$,
$$
M(\Gamma)=M(\sigma(\Gamma))=\gamma_n\left(\frac{t/2}{|\sigma(z)-\sigma(0)|}\right)
$$
and
$$
M(\Gamma')\geq M(\Delta([y,z],P_t;H))=\gamma_n\left(\frac{t/2}{|\sigma(z)-\sigma(y)|}\right).
$$
By $K-$quasiconformality we have $K\,M(\Gamma)\geq M(\Gamma')$ implying
$$
\frac{|f(x)-z|}{2t+|f(x)-z|}\leq\varphi_{K,n}\left(\frac{|x-z|}{2t+|x-z|}\right).
$$
Setting $2t/|x-z|=(1-s)/s$, we have
$$
\frac{|f(x)-z|}{|x-z|}\leq \frac{1-s}{s}\frac{\varphi_{K,n}(s)}{1-\varphi_{K,n}(s)}.
$$
Since $D$ is convex, it is easy to see that
$$|f(x)-z|^2\geq |x-f(x)|^2+|x-z|^2,$$
and hence
$$
\dfrac{|x-f(x)|}{|x-z|}\leq\sqrt{\left(\dfrac{|f(x)-z|}{|x-z|}\right)^2-1}\,.
$$
The definition of the $j-$metric, together with the last two inequalities yields
$$
j_D(x,f(x))\leq \log\left(1+\sqrt{\left(\frac{1-s}{s}\frac{\varphi_{K,n}(s)}{1-\varphi_{K,n}(s)}\right)^2-1}\right).
$$
Taking $s=1/3$, i.e. $t=|x-z|$, we get the inequality as desired.
\hfill$\Box$
}
\end{nonsec}

In order to prove Theorem \ref{thm:bnd4convexSK}, we need the following lemma.

\blem\label{lem4log}
The function
$$h_1(t)=\dfrac{\log^2(1+t)}{\log(1+t^2)}$$
is strictly decreasing in $(0,1)$ and strictly increasing in $(1,\infty)$. In particular, for $0<t\leq1$,
\begin{equation}\label{bernoulli}
\log^2(1+t)\leq\log(1+t^2).
\end{equation}
\elem

\bpf
Let $g(t)=(1+1/t)\log(1+t)=g_1(t)/g_2(t)$ with $g_1(t)=\log(1+t)$ and $g_2(t)=t/(1+t)$. It is easy to see that
$g_1(0)=0=g_2(0)$ and $g'_1(t)/g'_2(t)=1+t$ which is clearly strictly increasing in $(0,\infty)$. It follows from the l'H\^opital Monotone Rule \cite[Lemma 2.2]{avv1} that the function $g$ is strictly increasing in $(0,\infty)$.

By elementary computation, we have
$$
h'_1(t)=\frac{2t^2\log(1+t)}{(1+t)(1+t^2)\log^2(1+t^2)}(g(t^2)-g(t)),
$$
which is negative for $t\in(0,1)$ and positive for $t\in(1,\infty)$ by the monotonicity of $g$. Hence $h_1$ is strictly decreasing in $(0,1)$ and strictly increasing in $(1,\infty)$. The inequality (\ref{bernoulli}) follows from the monotonicity of $h_1$ since $h_1(0+)=1$.
\epf

\begin{nonsec}{\rm {\bf Proof of Theorem \ref{thm:bnd4convexSK}.}
Since $\lambda_n\leq 2e^{n-1}$ and $K\leq K_n$, it follows that
\begin{equation}\label{eqn:pfbndconvSK1}
(3\lambda_n)^{1-\alpha}\leq2.
\end{equation}
It is easy to check that, for $1\leq a\leq2$,
$$
\frac{2a/3}{1-a/3}\leq a^2,
$$
which, together with \eqref{eqn:pfbndconvSK1} and \eqref{bd4phi}, implies that
\begin{equation}\label{eqn:pfbndconvSK2}
\frac{2\varphi_{K,n}(1/3)}{1-\varphi_{K,n}(1/3)}\leq\frac{2(3\lambda_n)^{1-\alpha}/3}{1-(3\lambda_n)^{1-\alpha}/3}
\leq(3\lambda_n)^{2(1-\alpha)}\leq4.
\end{equation}
Let $h_1$ be as in Lemma \ref{lem4log}. Then
   $$h_1\left(\sqrt{\left(\frac{2\varphi_{K,n}(1/3)}{1-\varphi_{K,n}(1/3)}\right)^2-1}\right)\leq \max\{h_1(0+),h_1(\sqrt{15})\}=h_1(0+)=1$$
   since $h_1(\sqrt{15})=0.9046\cdots<1$. Hence
\begin{eqnarray*}\label{eqn:pfbndconvSK3}
\log\left(1+\sqrt{\left(\frac{2\varphi_{K,n}(1/3)}{1-\varphi_{K,n}(1/3)}\right)^2-1}\right)
&\leq&\sqrt{2\log\frac{2\varphi_{K,n}(1/3)}{1-\varphi_{K,n}(1/3)}}\\
&\leq&2\sqrt{\log{(3\lambda_n)}}\sqrt{1-\alpha}\\
&\leq&2\sqrt{\log{\left(6^{1-1/K}K\right)}}\\
&\leq&2\sqrt{1+\log6}\sqrt{K-1},
\end{eqnarray*}
as desired.
\hfill$\Box$
}
\end{nonsec}

Since a bounded convex domain is uniform, we have the following estimate for the quasihyperbolic metric.

\bcol\label{col4kd}
Let $D\subsetneq\Rn$ be a bounded convex domain and $f\in{\rm Id}_K(\partial D)$. Then for all $x\in D$.
$$
k_D(x,f(x))\leq U(D)\log\left(1+\sqrt{\left(\frac{2\varphi_{K,n}(1/3)}{1-\varphi_{K,n}(1/3)}\right)^2-1}\right)
$$
where $U(D)$ is the uniformity constant of the domain $D$.
\ecol

 It is well known that the unit ball $\Bn$ is a uniform domain with the constant $U(\Bn)=2$. This fact,
 together with Corollary \ref{col4kd} and Theorem  \ref{thm:bnd4convexSK}, yields the following estimate.

\bcol
Let $K\in(1,K_n]$ and $f\in{\rm Id}_K(\partial \Bn)$, then for all $x\in \Bn$
$$
k_{\Bn}(x,f(x))\leq 4\sqrt{1+\log6}(K-1)^{1/2}.
$$
\ecol

\begin{remark}{\rm
For the planar case, by \cite[Theorem 10.5 (3)]{avvb} we have
$$
1-\varphi_{K,2}(r)=(1+\varphi_{K,2}(r))\varphi_{1/K,2}\left(\frac{1-r}{1+r}\right),
$$
which, together with (\ref{bd4phi}) and (\ref{bd4phi2}), yields for $K>1$
\begin{equation}\label{inequal:lowerbnd41minusphi}
1-\varphi_{K,2}(r)\geq4^{1-K}(1+r)^{1-K}(1-r)^K\geq8^{1-K}(1-r)^K.
\end{equation}
By using the inequality (\ref{inequal:lowerbnd41minusphi}) and \eqref{bd4phi}, a similar argument as in the proof of Theorem \ref{thm:bnd4convexSK} gives
$$
\log\left(1+\sqrt{\left(\frac{2\varphi_{K,2}(1/3)}{1-\varphi_{K,2}(1/3)}\right)^2-1}\right)\leq\sqrt{2\log12}\sqrt{K-1/K}.
$$
}
\end{remark}

\begin{nonsec} {\rm {\bf Proof of Theorem \ref{thm4unipfbnd}.}
The idea of this proof is exactly the same as in the case of uniform domain with connected boundary \cite{vu84}.
Write $y=f(x)$. We may assume $d(x)\leq d(y)$. Fix  $z\in \partial D$ such that $d(x)=|x-z|$. Then we have
$|y-z|\geq|x-z|$ and
\beqn\label{ratio1}
\dfrac{|y-z|}{|x-z|}\geq\dfrac12\dfrac{|x-y|}{d(x)}\geq\frac12\left(e^{j_D(x,y)}-1\right)\geq \frac12\left(e^{k_D(x,y)/U}-1\right),
\eeqn
where $U$ is the uniformity constant of the domain $D$.
Assume now that $k_D(x,y)\geq2nU\log(1+2e^s)$ where $s$ is the constant of uniform perfectness of the domain $D$. Then this condition together with (\ref{ratio1}) yields
\beqn\label{ratio2}
\dfrac{|y-z|}{|x-z|}\geq e^{k_D(x,y)/(2U)}\geq(1+2e^s)^n.
\eeqn
Write $m=|x-z|$, $M=|y-z|$ and $t=m^{1/n}M^{1-1/n}$. Then $t\in(m,M)$ and $M/t=(M/m)^{1/n}\geq 1+2e^s$.
Let $[a,x]=\{au+x(1-u):\, 0\leq u\leq1\}$, $A=\partial D\setminus\overline\Bn(x,t)$, and let $\Gamma_t=\Delta([a,x],A)$ be the family of all curves joining $[a,x]$ to $A$. From Lemma \ref{lemma} and \cite[7.5]{va} it follows that
$$
C\log\dfrac{M}{t}\leq M(f\Gamma_t)\leq K\,M(\Gamma_t)\omega_{n-1}\left(\log\dfrac{t}{m}\right)^{1-n},
$$
and hence
\beqn\label{lbnd}
  K \geq d_n\left(\log\dfrac{M}{m}\right)^n;\quad d_n=C(n,s)(n-1)^{n-1}/(\omega_{n-1}n^n),
\eeqn
where $C(n,s)$ depends only on $s$ and $n$, and $\omega_{n-1}$ depends only on $n$.
Combining (\ref{lbnd}) and the first inequality of (\ref{ratio2}), we have
$$
  K \geq d_n(2U)^{-n}k_D(x,y)^n
$$
for $k_D(x,y)\geq 2nU\log(1+2e^s)$.
Since $K\geq1$, it is clear that
$$K\geq k_D(x,y)^n/(2nU\log(1+2e^s))^n$$
for $k_D(x,y)< 2nU\log(1+2e^s)$.
Hence in all cases
$$
  K \geq c_2(n,D)k_D(x,y)^n,
$$
where $c_2(n,D)=\min\{d_n(2U)^{-n},(2nU\log(1+2e^s))^{-n}\}.$
\hfill$\Box$
}
\end{nonsec}

Next we study the distortion of $K$--quasiconformal mappings $f:\,\MRn\to\MRn$  with the property
\begin{equation}
f(te_1)=te_1\qquad \mbox{for all $t\in\R$}.
\end{equation}
The following theorem improves the result of \cite[Theorem 1.6]{fv}. Observe that the constant in the theorem tends to one when $K$ goes to 1.

\begin{theorem}
Let $f:\,\Rn\to\Rn$ be a $K$--quasiconformal mapping which keeps the $x_1$--axis pointwise fixed. If $K>1$, then
$$
|f(x)|\leq\left(\lambda_n^{2\beta-2}\frac{\beta^{\beta}}{(\beta-1)^{\beta-1}}-\lambda^{2-2\beta}\frac{(\beta-1)^{\beta+1}}{4\beta^{\beta}}\right)|x|
$$
for all $x\in\Rn$ where $\beta=K^{1/(n-1)}$ and $\lambda_n$ is the Gr\"otzsch ring constant.
\end{theorem}

\begin{proof}
We may assume that $x\neq0$ and $f(x)$ is in the left half space (first coordinate non-positive).
For fixed $s>0$ let $h:\,\MRn\to\MRn$ be the M\"obius transformation which takes $f(x)$, $0$, $se_1$, $\infty$ onto
$-e_1$, $-y$, $y$, $e_1$, respectively, where $|y|<1$. We consider the ring $R'$ whose complement consists of
$E=h^{-1}([-y,y])$ and $F=h^{-1}([-e_1,\infty]\cup[e_1,\infty])$. By \cite[Theorem 15.9]{avvb}, we have
\begin{eqnarray}\label{inequal:capR}
{\rm cap}R'&\leq & \tau_n\left(\frac{|f(x)/s|+|f(x)/s-e_1|-1}{2}\right)\nonumber\\
           &\leq&\tau_n\left(\frac{|f(x)/s|+\sqrt{|f(x)/s|^2+1}-1}{2}\right),
\end{eqnarray}
where the second inequality follows from the inequality $|f(x)-se_1|^2\geq|f(x)|^2+s^2$ and the monotonicity of $\tau_n$. On the other hand, we put $R=f^{-1}(R')$ and conclude by \cite[Theorem 8.44]{avvb}
\begin{equation}\label{inequal:capR'}
{\rm cap}R\geq \tau_n\left(\frac{|x|}{s}\right).
\end{equation}
Inequalities (\ref{inequal:capR}), (\ref{inequal:capR'}), and ${\rm cap}R\leq K{\rm cap}R'$ then yield
$$
\frac{|f(x)|}{|x|}t+\sqrt{\left(\frac{|f(x)|}{|x|}t\right)^2+1}\leq 2\eta_{K,n}(t)+1,\qquad t=\frac{|x|}{s}.
$$
Hence
\begin{eqnarray}\label{inequal:dist}
\frac{|f(x)|}{|x|}&\leq&\frac{(2\eta_{K,n}(t)+1)^2-1}{2(2\eta_{K,n}(t)+1)t}\nonumber\\
                  &<&\frac{(\eta_{K,n}(t)+1)^2-1/4}{(\eta_{K,n}(t)+1)t}\nonumber\\
                  &\leq&\lambda_n^{2(\beta-1)}\frac{(1+t)^\beta}{t}-\frac{\lambda^{2(1-\beta)}}{4(1+t)^{\beta}t},
\end{eqnarray}
where (\ref{inequal:dist}) follows from the formula (\ref{eta}) and the second inequality in (\ref{bd4phi2}).
The choice of $t=1/(\beta-1)$ yields
$$
\frac{|f(x)|}{|x|}\leq \lambda_n^{2\beta-2}\frac{\beta^{\beta}}{(\beta-1)^{\beta-1}}-\lambda^{2-2\beta}\frac{(\beta-1)^{\beta+1}}{4\beta^{\beta}},
$$
and the theorem is proved.
\end{proof}

\begin{nonsec}{\rm {\bf Proof of Theorem \ref{holder}.}
Let us extend $f$ identically outside the unit ball.
For $R>1$  let $h(x)=x/R$, then
$$g:=h\circ f\circ h^{-1}:\Bn\to\Bn$$
is a $K$--quasiconformal mapping.
By applying the following well-known inequality \cite[3.1]{mrv} (also see \cite[3.3]{vu85})
$$\tanh\dfrac{\rho(f(x),f(y))}{2}\leq\varphi_{K,n}\left(\tanh\dfrac{\rho(x,y)}{2}\right)$$
and the estimate for the hyperbolic metric \cite[Exercise 2.52(1)]{vu88}
$$\dfrac{|x-y|}{1+|x||y|}\leq\tanh\dfrac{\rho(x,y)}{2}\leq\dfrac{|x-y|}{1-|x||y|}$$ to the mapping $g$ and points
$x/R,y/R$ for $x,y\in\Bn$,
we have
$$\dfrac{|f(x)/R-f(y)/R|}{1+|f(x)||f(y)|/R^2}\leq\varphi_{K,n}\left(\dfrac{|x/R-y/R|}{1-|x||y|/R^2}\right).$$
Hence by (\ref{bd4phi})
\begin{eqnarray}\label{hldineq}
|f(x)-f(y)|&\leq&\lambda_n^{1-\alpha}\dfrac{R+|f(x)||f(y)|/R}{(R-|x||y|/R)^\alpha}|x-y|^\alpha\nonumber\\
&\leq&\lambda_n^{1-\alpha}A(R)|x-y|^\alpha,
\end{eqnarray}
where $$A(R)=\dfrac{R+R^{-1}}{(R-R^{-1})^\alpha}.$$
It is easy to check that $A(1+)=\infty=A(\infty)$ and
$$R_0=\sqrt{\dfrac{1+\sqrt{\alpha}}{1-\sqrt{\alpha}}}$$
is the unique value of $R$ in the interval $(1,\infty)$ such that $A'(R)=0$. Hence we have
$$C(\alpha):=\min_{1<R<\infty}A(R)=A(R_0)=2^{1-\alpha}\alpha^{-\alpha/2}(1-\alpha)^{(\alpha-1)/2}.$$
Since the inequality (\ref{hldineq}) holds for all $R>1$, we get
$$
|f(x)-f(y)|\leq\lambda_n^{1-\alpha}C(\alpha)|x-y|^\alpha.
$$
It is easy to see that $C(1-)=1$, and hence $M_1(n,K)=\lambda_n^{1-\alpha}C(\alpha)\to1$ as $K\to1$.
\hfill$\Box$
}
\end{nonsec}

Applying the above theorem to the inverse of $f$, we have the following corollary.

\bcol
If $f\in{\rm Id}_K(\partial\Bn)$, then for all $x,y\in\Bn$
$$\dfrac1{M_2(n,K)}|x-y|^{1/\alpha}\leq|f(x)-f(y)|\leq M_2(n,K)|x-y|^\alpha$$
where $M_2(n,K)=M_1(n,K)^{1/\alpha}.$
\ecol

The following figure shows some upper bounds for Mori's constant, and it should be noted that
the first three bounds hold for all quasiconformal self-maps of the unit ball with the origin fixed but without the additional condition of identity boundary values, while the fourth bound holds for all quasiconformal self-maps of the unit ball with identity boundary values but without the condition of the origin fixed. Note that the figure shows the logarithms of the bounds.

\begin{figure}[H]
\includegraphics[width=9cm]{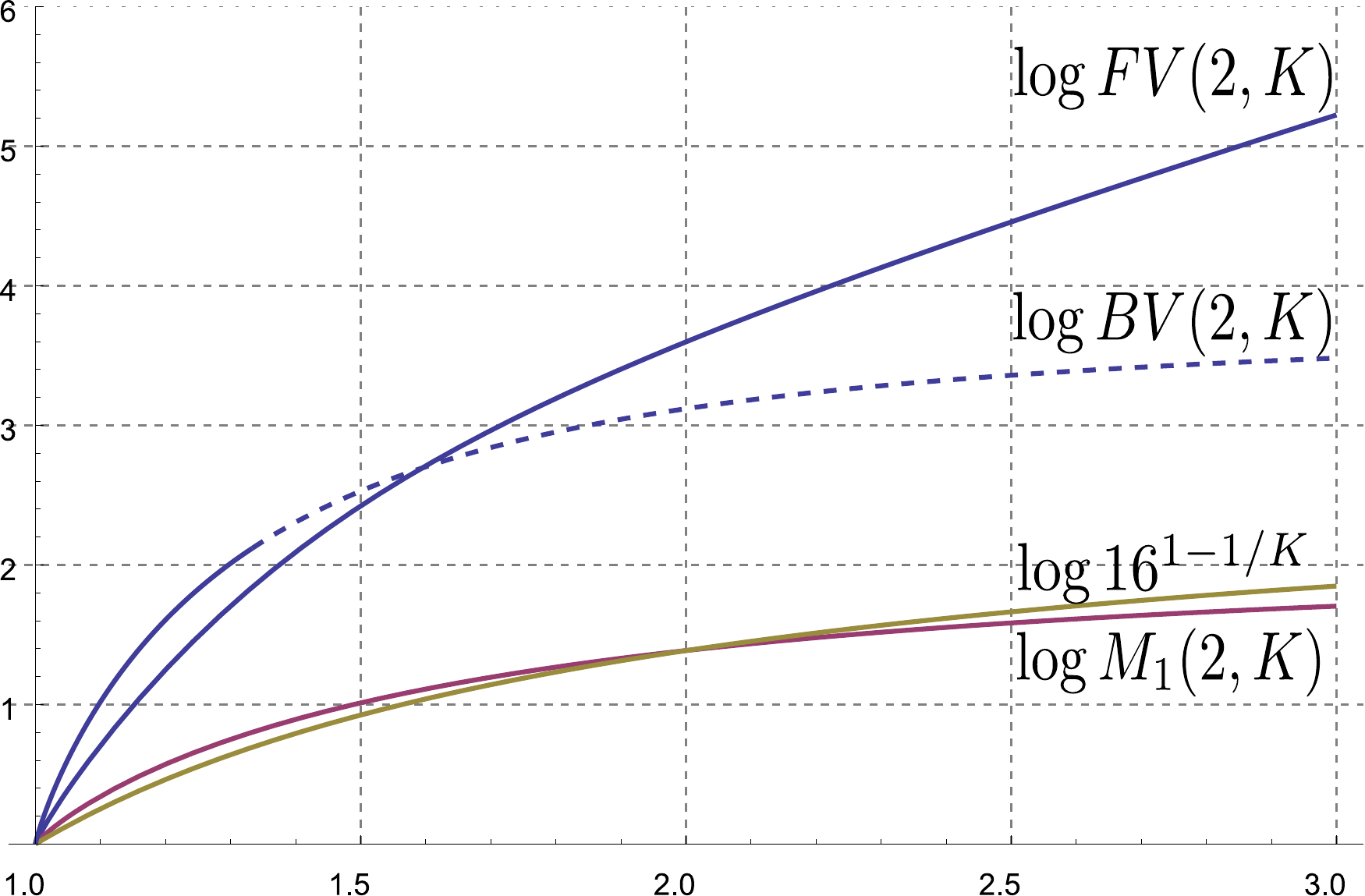}
\caption{Graphical comparison of various bounds when $n=2$ and $\lambda_2=4$, as a function of $K$: (a) the Fehlmann and Vuorinen bound \cite{fv}
$$
FV(2,K)=\left(1+\varphi_{2,K}\left(\dfrac{K^2-1}{K^2+1}\right)\right)2^{2K-3/K}
\dfrac{(K^2+1)^{(K+1/K)/2}}{(K^2-1)^{(K-1/K)/2}},
$$
(b) the Bhayo and Vuorinen bound \cite[1.8, for $n=2$]{bv}, valid for $K\in(1,K_1)$, $K_1=4/3$,
$$
BV(2,K)=3^{1-1/K^2}4^{1-1/K-1/K^2}K^2(K^2-1)^{1/K^2-1},
$$
(c) Mori's conjectured bound $16^{1-1/K}$, (d) the bound $M_1(2,K)$ from Theorem \ref{holder}.}
\end{figure}

\brmk
{\rm In \cite[Remark 3.15]{mv}, it is proved that
$$
  \varphi_{K,2}(r)\leq2\varphi_{K,2}\left(\sqrt{\dfrac{1+r}{2}}\right)^2-1
$$
for all $r\in[0,1]$. Writing $A(r,s)=\sqrt{(r+s)/2}$, then the authors conjectured that
\beqn\label{inequal4phi}
  A(\varphi_{K,2}(r),\varphi_{K,2}(s))\leq \varphi_{K,2}(A(r,s))
\eeqn
holds for all $r,s\in[0,1]$. In fact, by \cite{wzc}
$$
  A(\varphi_{K,2}(r),\varphi_{K,2}(s))^2\leq \varphi_{K,2}(A(r,s)^2),
$$
and by \cite[Theorem 10.15]{avvb}
$$
  \varphi_{K,2}(A(r,s)^2)^{1/2}\leq\varphi_{K,2}(A(r,s)).
$$
Now these two inequalities imply (\ref{inequal4phi}).}
\ermk

\medskip

\subsection*{Acknowledgments}
The research of Matti Vuorinen was supported by the Academy of Finland, Project 2600066611.
Xiaohui Zhang is indebted to the Finnish National Graduate School of Mathematics and
its Applications for financial support.



\end{document}